\documentclass[11pt]{article}
\usepackage{amsfonts}

\usepackage{graphics}
\usepackage{indentfirst}
\usepackage{amsthm}

\usepackage{cite}
\usepackage{latexsym}
\usepackage{amsmath}
\usepackage{amssymb}
\usepackage[dvips]{epsfig}
\usepackage{amscd}
\newcommand{\ba}{\begin{aligned}}
\newcommand{\ea}{\end{aligned}}
\newcommand{\be}{\begin{equation}}
\newcommand{\ee}{\end{equation}}
\newcommand{\bn}{\begin{eqnarray}}
\newcommand{\en}{\end{eqnarray}}
\newcommand{\bnn}{\begin{eqnarray*}}
\newcommand{\enn}{\end{eqnarray*}}
\newcommand{\ti}{\tilde}

\newcommand{\lm}{\lambda}
\renewcommand{\div}{ {\rm div }  }

\renewcommand{\O}{\Omega }
\def\rr{\mathbb{R}^3}

\newcommand{\na}{\nabla }

\newcommand{\pa}{\partial}
\newcommand{\bi}{\bibitem}

\newcommand{\bl}{\begin{lemma}}
\newcommand{\el}{\end{lemma}}
\newcommand{\et}{\end{theorem}}

\newcommand{\te}{\theta}

\newcommand{\ve}{\varepsilon}
\newcommand{\la}{\label}
\newcommand{\ka}{\kappa}
\newcommand{\om}{\Omega}

\newcommand{\ep}{\varepsilon}
\newcommand{\n}{\rho}

\def\la{\label}
\def\na{\nabla}

 \def\p{\partial}

\def\norm[#1]#2{\|#2\|_{#1}}
\def\nm[#1]#2{\|#2\|_{#1}}
\numberwithin{equation}{section}

 \usepackage{indentfirst}
\usepackage{graphics}
\usepackage{epsfig}

\newtheorem{claim}{\bf \t}[part]

\newtheorem{theorem}{Theorem}[section]

\newtheorem{lemma}[theorem]{Lemma}

\newtheorem{remark}{Remark}[section]
\newtheorem{definition}{Definition}[section]

\def\t{\theta}

\def\m{\mu}

\def\l{\lambda}
\def\rr{\mathbb{R}^3}
\def\r{\rho}

\def\f{\frac}
 \title{  On formation of singularity for non-isentropic Navier-Stokes equations without heat-conductivity }

\author{Xiangdi H{\small UANG}$^{a}$,  Zhouping X{\small IN}$^{b}$  \\[3mm] {\normalsize $^a$ NCMIS, AMSS,} \\
{\normalsize Chinese Academy of Sciences, Beijing 100190,P. R. China} \\[2mm]
{\normalsize $^b$ The Institute of Mathematical Sciences,} \\
{\normalsize  The Chinese University of Hong Kong, Hong Kong,
P. R. China}
 }

\date{}
\begin{document}
 \maketitle

\begin{abstract}
It is known that smooth solutions to the non-isentropic Navier-Stokes equations without heat-conductivity may lose their regularities in finite time in the presence of vacuum. However, in spite of the recent progress on such blowup phenomenon, it remain to give a possible blowup mechanism. In this paper, we present a simple continuation principle for such system, which asserts that the concentration of the density or the temperature occurs in finite time for a large class of smooth initial data, which is responsible for the breakdown of classical solutions. It also give an affirmative answer to a strong version of conjecture proposed by J.Nash in 1950s.

\


\

AMS: 35Q35, 35B65, 76N10
\end{abstract}

\section{Introduction}

In this paper, we consider the system of partial differential
equations for the three-dimensional  compressible, and non-isentropic Navier-Stokes equations
in the Eulerian coordinates
 \be  \label{a1}
\begin{cases}
\r_{t}+\mbox{div}(\r u)=0,\\
(\r u)_{t}+\mbox{div}(\r u\otimes
u)-\mu\Delta{u}-(\mu+\l)\nabla\mbox{div}u
+\nabla{P}= 0,\\
c_v[(\r \t)_{t}+\mbox{div}(\r u\t
)]-\kappa\triangle\theta+P\mbox{div}u=2\mu|\mathfrak{D}(u)|^2+\l(\mbox{div}u)^2,
\end{cases}
\ee
where $t\ge 0$ is time, $x\in \Omega\subset \rr$ is the spatial coordinate, and
 $\n,  u=\left(u_1,u_2,u_3\right)^{\rm tr},  $ $\te, $ $P=R\r\t \,(R>0), $
represent respectively the fluid density,  velocity, absolute temperature, pressure; $ \mathfrak{D}(u)$ is the
deformation tensor given by  \bnn \la{z1.2}   \mathfrak{D}(u)  =
\frac{1}{2}(\nabla u + (\nabla u)^{\rm tr}). \enn
 The constant
viscosity coefficients $\mu$ and $\lambda$  satisfy the physical
restrictions
 \be\la{a2} \mu>0,\quad 2 \mu + 3\lambda\ge 0.
\ee Positive constants $c_v,$  $\ka,$ and $\nu$ are respectively  the heat capacity,  the ratio of the heat conductivity coefficient over the heat capacity.

The compressible Navier-Stokes system \eqref{a1} consists of  a set of equations describing compressible viscous heat-conducting flows. Indeed, the equations \eqref{a1}$_1$, \eqref{a1}$_2,$ and \eqref{a1}$_3$
  respectively  describe the conservation of mass, momentum, and energy.

There is a considerable body of literature on the multi-dimensional compressible Navier-Stokes system \eqref{a1} by physicists and mathematicians
because of its physical importance, complexity, rich phenomena, and mathematical
challenges; see \cite{choe1,feireisl1,Hof1,hulx,hlx4,L1,M1,Na,R,se1,X1} and the references cited therein. However, many physically important and mathematically fundamental
problems are still open due to the lack of smoothing mechanism and the strong nonlinearity. For example, although
  the local strong solutions to the compressible
 Navier-Stokes system \eqref{a1} for general initial data with nonnegative density were respectively obtained by \cite{choe1},
whether the unique local strong solution   can exist globally in time is an outstanding challenging open problem in contrast to the isentropic case [13].

  In the presence of vacuum, as pointed out by Xin\cite{X1}, non-isentropic Navier-Stokes equations without heat-conductivity will develop finite time singularity, see also reference \cite{choe}. Indeed, very recently, Xin-Yan\cite{XY} further proved that any classical solutions of viscous compressible fluids with or without heat conduction will
blow up in finite time, as long as the initial data has an isolated mass group. Their results hold for the whole space and bounded domains, yet the blowup mechanism is not clarified. It is the main purpose if this paper to resolve this key issue. Theorem \ref{thm1.2} reveals that the concentration of the density or the temperature must be responsible for the loss of regularity in finite time.

  Although vacuum will lead to breakdown of smooth solutions in finite time , it is also important to study  the mechanism of blowup and structure of possible singularities of general strong (or smooth) solutions to the compressible  Navier-Stokes system.

The pioneering work can be traced   to Serrin's criterion \cite{se2}  on the Leray-Hopf weak solutions to the three-dimensional incompressible Navier-Stokes
equations, which can be stated that if a weak solution $u$ satisfies
\begin{equation}\label{1.6} u\in L^s(0,T;L^r),\quad
\frac{2}{s}+\frac{3}{r}\leq 1,\quad 3<r\leq \infty,
\end{equation}
then it is regular.

 Recently, Huang-Li-Xin \cite{hlx}  extended  the Serrin's  criterion  \eqref{1.6}
to the barotropic compressible Navier-Stokes
equations  and  showed  that  if $T^*<\infty$ is the maximal
time of existence of a strong (or classical) solution $(\n,u)$, then
\be\la{ba2} \lim_{T\rightarrow T^*}\left(
\|\rho\|_{L^\infty(0,T;L^{\infty})} + \| {
u}\|_{L^s(0,T;L^r)}\right) = \infty ,
 \ee    with $r$ and $s$ as in   (\ref{1.6}).  For more information  on the blowup criteria of barotropic compressible flow, we refer to
\cite{hlx,hlx1,H2,H4,wz} and the references therein.

When it comes to the full compressible Navier-Stokes system \eqref{a1}, the problem is more complicated. In\cite{Na-1}, Nash proposed a conjecture on the possible blowup of compressible heat-conductive flows. He wrote
  "This should give existence, smoothness, and continuation (in time) of flows, conditional on the non-appearance of certain gross type of singularity, such as infinities of temperature or density."

 Under the   condition that
\be\la{a6} \lambda<7\mu, \ee
  Fan-Jiang-Ou \cite{J2}  obtained   the following blowup criterion     \bnn
 \lim_{T\rightarrow T^*}(
\|{\theta}\|_{L^{\infty}(0,T;L^{\infty})} +
 \|{\nabla u}\|_{L^1(0,T;L^{\infty})})  = \infty .\enn
Recently,  under just the physical restrictions (\ref{a2}),    Huang-Li
\cite{hl} and Huang-Li-Xin \cite{hlx1}   established the
following  blowup criterion: \bnn \lim_{T\rightarrow
T^*}\left(
 \|{\theta}\|_{L^2(0,T;L^{\infty})}+\|{\mathfrak{D}( u)}
 \|_{L^1(0,T;L^{\infty})}\right)  = \infty ,\enn
where $ \mathfrak{D}(u)$ is the deformation tensor.

Later, in the absence of vacuum,
Sun-Wang-Zhang\cite{wz1} established the following blowup criterion for bounded domains  with positive heat-conductivity $\kappa>0$ that
\be \la{fns2-1}\lim_{T\rightarrow
T^*}(\norm[L^\infty(0,T;L^{\infty})]{(\n,\frac{1}{\n})} +
\norm[L^\infty(0,T;L^\infty)]{\t}) = \infty,
 \ee
  provided that  \eqref{a2} and  (\ref{a6}) hold true. As a consequence, Nash's conjecture is partially verified as \cite{wz1} can't rule out the possibility of appearance of vacuum.

     Recently, for $\kappa>0$, we\cite{HL-C} establish a blowup criterion allowing initial vacuum, which  is independent of temperature, as follows
 \be \la{fns2-2}\lim_{T\rightarrow
T^*}(\norm[L^\infty(0,T;L^{\infty})]{\n} +
\norm[L^r(0,T;L^s)]{u}) = \infty,
 \ee
 where $r,s$ satisfy \eqref{1.6}.

 As a matter of fact, the blowup criterion \eqref{fns2-2} further implies
 \be \la{fns2-3}\lim_{T\rightarrow
T^*}(\norm[L^\infty(0,T;L^{\infty})]{\n} +
\norm[L^\infty(0,T;L^{\infty})]{\theta}) = \infty.
 \ee
as long as  \eqref{a6} holds true.
 This makes Nash's conjecture as an immediately corollary for positive heat-conductivity flows. Our main goal in this paper is to give an affirmative answer to a strong version of Nash's conjecture without heat-conduction.

We will assume that $\kappa=0$, and without loss of generality, take $c_v=R=1$. The system (\ref{a1}) is reduced to
\be  \label{n2}
\begin{cases}
\r_{t}+\mbox{div}(\r u)=0,\\
(\r u)_{t}+\mbox{div}(\r u\otimes
u)-\mu\Delta{u}-(\mu+\l)\nabla\mbox{div}u
+\nabla{P}= 0,\\
P_{t}+\mbox{div}(P u
)+P\mbox{div}u=2\mu|\mathfrak{D}(u)|^2+\l(\mbox{div}u)^2.
\end{cases}
\ee

The system (\ref{n2}) is supplemented with the following initial conditions:
\be\la{a3}  (\rho,u,P)(x,0)=(\rho_0, u_0,P_0)(x),\quad x\in R^3,\ee
$(\rho,u,P)$ satisfies  the far field
condition:
\begin{equation}\label{bc1}
 (\r,u,P)(x,t)\rightarrow (0,0,0)
 ~~\mbox{as}~~|x|\rightarrow\infty;
\end{equation}

To state the main result, we will use the following notations and conventions.

{\it Notations.} For $1\le p\le \infty$ and integer $k\ge 0$, the standard homogeneous and inhomogeneous Sobolev
spaces in $R^3$ are denoted by:
   \bnn \begin{cases} L^p=L^p(R^3),\quad W^{k,p}=W^{k,p}(R^3),\quad
    D^{k,p}  =
    \left.\left\{u\in
L^1_{\rm loc}(R^3)\,\right| {\nabla^k u}\in {L^p}  \right\},\\
  D^1_0   = \left. \left\{u\in L^6 \,\right|
  {\nabla u}\in{L^2} \right\},\quad H_0^1=L^2\cap D^1_0, \quad H^k=W^{k,2}
 \end{cases}\enn
Denote by
$$\dot{f}=f_t +u\cdot\na f,\quad\int fdx=\int_{R^3}fdx.$$
The strong solutions to the
  Cauchy problem \eqref{n2}-\eqref{bc1} are defined as follows.
\begin{definition}[Strong Solutions] \label{def1.1}
$(\r,u,P)$ is called a strong solution to \eqref{n2} in
$R^3\times (0,T)$, if for some $q_0>3$,
\begin{eqnarray}\nonumber
\begin{cases}
(\r,P)\geq 0,~~ (\r,P)\in C([0,T];L^1\cap H^1\cap W^{1,q_0}),~~(\r_t,P_t)\in
C([0,T];L^{q_0}),\\
u\in C([0,T];D_0^1\cap D^{2,2})\cap L^2(0,T;D^{2,q_0}),\quad  \\
u_t\in L^2(0,T;D^{1,2}), ~~
\sqrt{\r}u_t\in L^{\infty}(0,T; L^2),
\end{cases}
\end{eqnarray}
and $(\r,u,P)$ satisfies both \eqref{a1} almost everywhere in $R^3\times
(0,T)$ and  \eqref{a3} almost everywhere in $R^3.$

\end{definition}

Then the main result in this paper can be stated as follows:

\begin{theorem}\label{thm1.2}
 For constant $\tilde{q}\in(3,6]$, assume that  $(\r_0\ge
0,u_0,P_0\ge 0)$ satisfies
\begin{equation}
\ba\la{ini-1}
  (\r_0,P_0)\in L^1\cap H^1\cap W^{1,\tilde{q}},\quad
  u_0 \in D_0^1\cap
D^{2,2}, \quad    \r_0|u_0|^2\in L^1,
\ea
\end{equation}
and the compatibility conditions:
\be \label{n1.1}
 -\m\Delta{u}_0-(\m+\l)\nabla {\rm div}u_0+\nabla P_0 =\sqrt{\r_0}g,\ee
with $g\in L^2(R^3)$.
 Let $(\r,u,P)$ be the strong solution to the
   compressible Navier-Stokes system \eqref{n2} in $R^3$.
     If $T^\ast<\infty$ is the
maximal time of existence, then

 \be \la{fns2}\lim_{T\rightarrow
T^*}(\norm[L^\infty(0,T;L^{\infty})]{\n} +
\norm[L^\infty(0,T;L^\infty)]{\t}) = \infty.
 \ee
 provided
 \be\la{mul}
 \mu>4\lambda.
 \ee

\end{theorem}

A few remarks are in order:

\begin{remark}
Under the conditions of Theorem \ref{thm1.2}, the local existence of the strong solutions was guaranteed in \cite{choe1}. Thus, the assumption $T^*$ makes sense.
\end{remark}

\begin{remark}

 The main contribution of Theorem \ref{thm1.2} asserts that Nash's conjecture even holds for zero heat-conductivity flows. In the case $\kappa=0$, the formation of singularity is only due to the concentration of either the density or temperature. In this sense, we give an affirmative answer to a strong version of Nash's conjecture.
\end{remark}

\begin{remark}
  Condition (\ref{mul}) is only used in obtaining estimates in Lemma \ref{lem3.1}.
\end{remark}

\begin{remark}
  It's easy to prove a same continuation principle for two-dimensional problem without any restrictions on $\mu,\lambda$. Since the proof is analogous and simpler, we omit it for simplicity.
\end{remark}

We may also investigate the following different boundary conditions.

(1) $\Omega=\mathbf{R}^3$ and constants $\tilde \n,\tilde P\ge 0,$ $(\rho,u,P)$ satisfies  the far field
condition either vacuum or non-vacuum:
\begin{equation}\label{bc1-1}
 (\r,u,P)(x,t)\rightarrow (\tilde \r,0, \tilde P)
 ~~\mbox{as}~~|x|\rightarrow\infty;
\end{equation}
and initial condition
\be\la{ini-2}
(\r_0-\tilde\r,P_0-\tilde P)\in L^1\cap H^1\cap W^{1,\tilde{q}},\quad
  u_0 \in D_0^1\cap
D^{2,2}, \quad    \r_0|u_0|^2\in L^1,
\ee

(2) $\Omega$ is a bounded smooth domain.
   \be \label{bc2-1}u=0 \quad\mbox{ on }\partial\Omega.\ee
 For constant $\tilde{q}\in(3,6]$,  assume that $(\rho_0\ge 0,u_0,P_0\ge 0)$ $u$ satisfies
\be\la{ini-3}
(\r_0,P_0)\in L^1(\Omega)\cap W^{1,\tilde{q}}(\Omega),\quad
  u_0 \in H_0^1(\Omega)\cap H^2(\Omega).
\ee
(3)  $\Omega=\mathbf{T}^3=\mathbf{R}^3/\mathbf{Z}^3$ and constant $\tilde{q}\in(3,6]$, assume that $(\rho_0\ge 0,u_0,P_0\ge 0)$ satisfy
\be\la{ini-4}
(\r_0,P_0)\in L^1(\mathbf{T}^3)\cap W^{1,\tilde{q}}(\mathbf{T}^3),\quad
  u_0 \in H^2(\mathbf{T}^3).
\ee
Our next theorem asserts that Nash's conjecture also holds for different boundary conditions.
\begin{theorem}\label{thm1.3}
 Let $(\r,u,P)$ be the strong solution to the
  full compressible Navier-Stokes system \eqref{n2}   together with
    \be \label{n1.3-1} (\n,u,P)(x,0)=(\r_0,u_0,P_0), \quad x\in \om,\ee
    and
    \begin{enumerate}
    \item $(\r_0,u_0,P_0)$ satisfy \eqref{bc1-1}-\eqref{ini-2} for $\Omega=\mathbf{R}^3$,
    \item $(\r_0,u_0,P_0)$ satisfy \eqref{bc2-1}-\eqref{ini-3}
    for bounded domain $\Omega$,
    \item $(\r_0,u_0,P_0)$ satisfy \eqref{ini-4} for periodic domain $\Omega$,
    \end{enumerate}

and the compatibility conditions:
\be \label{n1.1}
 -\m\Delta{u}_0-(\m+\l)\nabla {\rm div}u_0+\nabla P_0 =\sqrt{\r_0}g,\ee
with $g\in L^2(\Omega)$.
     If $T^\ast<\infty$ is the
maximal time of existence, then

 \be \la{fns2}\lim_{T\rightarrow
T^*}(\norm[L^\infty(0,T;L^{\infty})]{\n} +
\norm[L^\infty(0,T;L^\infty)]{\t}) = \infty.
 \ee
 provided
 \be\la{mul}
 \mu>4\lambda.
 \ee

\end{theorem}

We now make some comments on the analysis of this paper.

Let $(\rho,u,\t)$ be a strong
solution described in Theorem \ref{thm1.2}. Suppose that
\eqref{fns2} were false,  that is,
\begin{equation}\label{1.10-1-1}
\lim_{T\rightarrow
T^*}(\norm[L^\infty(0,T;L^{\infty})]{\n} +
\norm[L^\infty(0,T;L^\infty)]{\t}) = M_0<+\infty.
\end{equation}
One needs to show that
\be\la{result-1}
\ba
&\sup_{0\le t\le T^\ast}\left(\| (\r,P)\|_{H^1\cap
W^{1,\tilde{q}}}+\|\nabla{u}\|_{H^1}\right) \le C.
\ea
\ee
Higher order derivatives estimates for above quantities then follow easily from above regularities.

Let's say a few words on the regularity criterion (\ref{fns2}). In the absence of heat-conductivity, the equation for the temperature changes its form from parabolic to hyperbolic type, thus resulting in loss of regularity benefiting from the smooth effect of heat dissipation. But fortunately, it enjoys a same nonlinear structure as the density equation. Since  the  methods in all previous works \cite{HLW-1,wz1} depend crucially on Hoff's a priori estimates. The main point is how to avoid  terms involving
density gradient in calculations. It turns out to be possible to treat the terms arising from pressure gradient. Some new ideas are needed to recover all the a priori estimates, that is, instead of the temperature $\te$ and pressure $P$, we treat  the total energy $E=\frac{1}{2}|u|^2+\t$.

 Finally, the   a priori estimates  on both the
$L^\infty_tL^p_x$-norm of the density gradient and pressure gradient along with the $L^1_tL^\infty_x$-norm of the velocity
gradient  can be obtained   simultaneously  by
solving a logarithm Gronwall inequality based on a  logarithm  estimate (see  Lemma \ref{lem2.3}) and the a priori estimates we have just derived.

The rest of the paper is organized as follows: In the next section, we
collect some elementary facts and inequalities that will be needed
later. The main result, Theorem \ref{thm1.2},  is proved in Section \ref{sec3}.
\section{Preliminaries}
In this section, we recall some known facts and elementary
inequalities that will be used later.

First,  the following
existence and uniqueness of local strong solutions when the initial
density may not be positive and may vanish in an open set  can be
found in \cite{choe1}.
\begin{lemma}\label{lem2.1}
Assume that the initial data $(\r_0\ge 0,u_0,P_0\ge 0)$ satisfy
\eqref{ini-1}-\eqref{n1.2}. Then there exists a positive time
$T_1\in(0,\infty)$ and a unique strong solution $(\r,u,P)$ to the
 Cauchy problem \eqref{a1}--\eqref{a3}  on $R^3\times
(0,T_1]$.
\end{lemma}

Next, the following well-known Sobolev inequality will be used later
frequently (see \cite{nir}).
\begin{lemma}\label{lem2.2}
For $p\in(1,\infty)$ and $q\in(3,\infty)$, there exists a generic
constant $C>0$, which depends only on $p,~q$ such that for $f\in D_0^1$
and $g\in L^p\cap D^{1,q}$, we have
\begin{equation}\label{2.1}
\|f\|_{L^6}\leq C\|\nabla f\|_{L^2},~~\|g\|_{L^\infty}\leq
C\|g\|_{L^p}+C\|\nabla g\|_{L^q}.
\end{equation}
\end{lemma}

The following  logarithm  estimate from\cite{hlx} will be used   to estimate $\|\nabla
u\|_{L^\infty}$ and $\|(\nabla\r,\na P)\|_{L^2\cap L^q}.$

\begin{lemma}\label{lem2.3}
For $3<q<\infty,$ there is a
constant  $C(q)$ such that  the following estimate holds for all
$\na u\in L^2 \cap D^{1,q} ,$ \be\la{ww7}\ba \|\na u\|_{L^\infty
}&\le C\left(\|{\rm div}u\|_{L^\infty }+ \|\na\times u\|_{L^\infty }
\right)\log(e+\|\na^2 u\|_{L^q })\\&\quad+C\|\na u\|_{L^2 } +C .
\ea\ee
\end{lemma}

Finally, we consider  the following Lam\'e   system \be\la{u89} -\mu \Delta v(x)-(\mu+\lm)\na \div v(x)=  f(x) ,\quad   x\in\om, \ee where $v=(v_1,v_2,v_3),$   $f=(f_1,f_2,f_3),$ and $\mu,\lm$ satisfy \eqref{a2}.

Assume that $\om$ is a bounded smooth domain
 and \be \la{b2} v=0 \mbox{ on }{\p\om}.\ee

The following  logarithm  estimate for the Lam\'e system \eqref{u89}, which can be found in \cite{HL-C} will be used   to estimate $\|\nabla
u\|_{L^\infty}$ and $\|\nabla\r\|_{L^2\cap L^q}.$

\begin{lemma}\label{lem2.3} Let $\mu,\lm$ satisfy \eqref{a2}.
Assume that  $f=\div g$ where $g=(g_{kj})_{3\times 3} $ with $g_{kj}\in L^2\cap W^{1,q}$ for    $k,j=1,\cdots,3, $  $r\in (1,\infty), $  and $q\in (3,\infty).$  Then the  Lam\'e system  \eqref{u89} together with \eqref{b2}   has a unique solution $v\in   H_0^1\cap W^{2,q},$ and  there exists a generic positive constant $C$ depending only on $\mu,\lm,q, $ $r$ and $\Omega$ such that
\be \la{b6}\|\na  v\|_{L^r}\le C\|g\|_{L^r},\ee
and \be\la{ww7} \|\na v\|_{L^\infty}\le C \left(1+\ln
(e+\|\na g\|_{L^q})\|g\|_{L^\infty}+ \|g\|_{L^r}\right).\ee

\end{lemma}
\section{\la{sec3}Proof  of Theorem \ref{thm1.2}.}
Throughout the rest of the section, $C$ will denote a generic constant depending only on $\rho_0,u_0,\t_0,T^*,M_0,\lambda,\mu$.

We start with the standard energy estimate
\begin{lemma}
\be\la{energy}
\int\left(\frac{1}{2}\rho|u|^2+P\right)dx = \int\left(\frac{1}{2}\rho_0|u_0|^2+P_0\right)dx \triangleq E_0.
\ee
\end{lemma}

Next, a high energy estimate holds under the condition \eqref{1.10-1-1}.

\begin{lemma}\la{lem3.1} Under the condition \eqref{1.10-1-1}, as long as $\mu>4\lambda$, it holds that
\be
\int\rho(|u|^2+|u|^6)dx + \int_0^t\int|\na u|^2(1+|u|^2 + |u|^4)dxdt\le C
\ee
\end{lemma}

\begin{proof} It follows from \eqref{1.10-1-1} that
\be
\|\rho\|_{L^1\cap L^\infty}\le C.
\ee
Multiplying $(1.1)_2$ by $q|u|^{q-2}u$, and integrating over $\O$, one obtains
by using lemma $2.1$ that
\be\la{eee-1-1}
\ba
& \frac{d}{dt}\int \rho|u|^qdx + \int (q|u|^{q-2}[\mu|\nabla u|^2 +
(\lambda + \mu)(\rm{div}u)^2 + \mu(q-2)|\nabla|u||^2] \\
& + q(\lambda + \mu)(\nabla|u|^{q-2})\cdot u\rm{div}u)dx \\
& = q\int P\rm{div}(|u|^{q-2}u)dx\\
& \le C\int \rho^{\frac{1}{2}}|u|^{q-2}|\nabla u|dx\\
& \le\ep\int |u|^{q-2}|\nabla u|^2dx + C(\ep)\int \rho|u|^{q-2}dx\\
& \le\ep\int |u|^{q-2}|\nabla u|^2dx + C(\ep)(\int \rho|u|^qdx)^{\frac{q-2}{q}}.
\ea
\ee

Noting that $|\nabla |u||\le |\nabla u|$, one gets that
\be\la{qqq}
\ba
& q|u|^{q-2}[\mu|\nabla u|^2 + (\lambda + \mu)(\rm{div}u)^2 + \mu(q-2)|\nabla|u||^2]
 + q(\lambda + \mu)(\nabla|u|^{q-2})\cdot u\rm{div}u \\
 & \ge q|u|^{q-2}[\mu|\nabla u|^2 + (\lambda + \mu)(\rm{div}u)^2 + \mu(q-2)|\nabla|u||^2\\
 & - (\lambda + \mu)(q-2)|\nabla|u||\cdot|\rm{div}u|]\\
 & = q|u|^{q-2}[\mu|\nabla u|^2 + (\lambda + \mu)(|\rm{div}u| - \frac{(q-2)}{2}|\nabla|u||)^2]\\
 & + q|u|^{q-2}[\mu(q-2) - \frac{1}{4}(\lambda + \mu)(q-2)^2]|\nabla|u||^2\\
 & \ge q|u|^{q-2}[\mu(q-1) - \frac{1}{4}(\lambda + \mu)(q-2)^2]|\nabla|u||^2
\ea
\ee
Consequently, recall $q=6$ and $\mu>4\lambda$, the left-hand side of \eqref{qqq} is greater than
 \be
 6(\mu-4\lambda)|u|^4|\nabla|u||^2.
 \ee
 Hence, taking $\epsilon$ small enough in \eqref{eee-1-1} and Gronwall's inequality implies
\be
\int\rho|u|^6dx + \int_0^t\int|\na u|^2|u|^4dxdt\le C.
\ee
Taking $q=4$ again, one can also prove
\be
\int\rho|u|^4dx + \int_0^t\int|\na u|^2|u|^2dxdt\le C.
\ee
This completes the proof.
\end{proof}

Also, $\theta$ is always non-negative before blowup time $T^*$.
\begin{lemma}
As long as $\t_0\ge 0$, it holds that
\be\la{theta}
  \inf_{(x,t)\in R^3\times(0,T^*)}\t(x,t)\ge 0.
\ee
\end{lemma}
\begin{proof}
  The equation for $P$ can be rewritten as
  \be
  P_t + div(Pu) + Pdiv u = F\ge 0.
  \ee
  Since the solution is smooth , we can always define particle path before blowup time.
   \be
   \left\{
   \ba
   &\frac{d}{dt}X(x,t)=u(X(x,t),t),\\
   & X(x,0)=x.
   \ea
   \right.
   \ee

   Consequently, along particle path, one has
  \be
  \frac{d}{dt}P(X(x,t),t) = -2Pdiv u + F
  \ee
  implies
  \be\la{pressure}
  P(X(x,t),t)=\exp(-2\int_0^tdiv uds)\left[P(0) +
  \int_0^t\exp(2\int_0^sdiv ud\tau)Fds\right]\ge 0.
  \ee
  Hence, $\theta\ge 0$ follows immediately from (\ref{pressure}).
\end{proof}

 Before proving Theorem \ref{thm1.2}, we  state some a priori estimates under the condition (\ref{1.10-1-1}).

 Let   $E$ be the  specific  energy defined by \be\la{energy}
E\triangleq  \te+\frac{|u|^2}{2}.\ee
Let $G,\omega$ be the effective viscous flux, vorticity respectively given by
\be
G=(2\m+\lambda)\mbox{div} u-P,\quad\omega=\mbox{curl} u.
\ee
Then, the momentum equations can be rewritten as
\be\la{mom}
\rho\dot{u} =\nabla G - \m\na\times\omega.
\ee

Then, we derive the following crucial  estimate on the $L^\infty(0,T;
L^2)$-norm of $\nabla u.$
\begin{lemma}\label{lem3.3}
Under the condition \eqref{1.10-1-1}, it holds that for $0\leq
T<T^\ast$,
\be\label{3.4}\sup_{0\leq t\leq T}
\int|\nabla u|^2dx+\int_{0}^{T}\int\r|\dot{u}|^2dxdt \leq C.
\ee
\end{lemma}

\begin{proof}
 First, multiplying $(\ref{a1})_2$ by $u_t$ and integrating the resulting
equation over $\O$ show that
\be \label{3.14}\ba
& \f12\f{d}{dt}\int\left(\m|\nabla u|^2+(\m+\l)(\mbox{div}u)^2\right)dx+\int\r|\dot{u}|^2dx\\
&= \int\r\dot{u}\cdot(u\cdot\nabla){u} dx+\int P\mbox{div}u_tdx\\
&  \le\f14\int\r|\dot{u}|^2dx+C\int\n|u|^2|\nabla u|^2dx+\f{d}{dt}\int P\mbox{div}udx -\int P_t\mbox{div}udx\\
& =\f14\int\r|\dot{u}|^2dx+C\int\n|u|^2|\nabla u|^2dx\\
&\quad\quad+\f{d}{dt}\int \left(P\mbox{div}u-\frac{1}{2(2\mu+\lambda)}P^2\right)dx -\frac{1}{2(2\mu+\lambda)}\int P_tGdx
\ea\ee

Then, we will estimate the last term on the righthand side of \eqref{3.14}.

First, it follows from $\eqref{a1} $   that $E$ satisfies
\begin{equation}\label{en-1}
(\r E)_t+\mbox{div}(\r Eu) =\div F
\ee
with
\be
\ba
F& \triangleq  \frac{\mu}{2}\na(|u|^2) +\m u\cdot\nabla u +\l u\mbox{div}u    - Pu.
\ea
\end{equation}

 It follows from \eqref{energy} and \eqref{en-1} that
\be \label{3.m1} \ba
   -\int P_t  G dx    & =-\int (\rho E)_t G dx+ \frac{1}{2}  \int (\rho |u|^2)_t G dx   \\
 & \le C   \int \left( \rho E |u| +|u||\na u|\right) |\na G| dx \\&\quad   - \frac{1}{2}\int\left( {\rm div} (\rho u) |u|^2 G       - 2 \rho u\cdot u_t  G \right)dx   =\sum\limits_{i=1}^3I_i.
\ea \ee

Cauchy's and Sobolev's inequalities yield that
\be\ba\la{mm-2}
 I_1+I_2&\le \eta\|\na G\|_{L^2}^2+C(\eta)\int\left(\n^2E^2|u|^2 + |u|^2|\na u|^2\right)dx\\
 &\le \eta\|\na G\|_{L^2}^2+C(\eta)\int\left(\n|u|^6 +\n|u|^2 + |u|^2|\na u|^2\right)dx.
 \ea\ee

 Integration by parts and recall $G=(2\mu+\lambda)\text{div}u -P$ also gives
\be\ba\la{mm-3} I_3&\le C\int \n |u|^3|\na G|dx+C\int \left(\n|u|^2|\na u|+\n |u||\dot u|\right)|G|dx\\
& \le C\int \n |u|^3|\na G|dx+C\int \left(\n|u|^2|\na u|+\n |u||\dot u|\right)(|\na u|+|P|)dx
\\ & \le  C(\eta)\int \left(\n |u|^6+\n |u|^2 + |u|^2|\na u|^2\right) dx +\eta\|\na G\|_{L^2}^2+\eta\int\n |\dot u|^2dx.\ea\ee

In view of \eqref{mom},
\be\la{mom-m}
\rho\dot{u} =\nabla G - \m\na\times\omega,
\ee
which implies
\be\la{mom-n}
\|\na G\|_2^2\le C\|\n\dot{u}\|_2^2\le C\int\rho|\dot{u}|^2dx.
\ee

Substituting  \eqref{3.m1}-\eqref{mm-3} and \eqref{mom-n}  into \eqref{3.14}, one obtains after choosing $\eta$ suitably small that
\be\la{zk1}\ba &\frac{d}{dt}\int\Phi dx +\int\n|\dot u|^2dx\\&\le  C \int \left(\n |u|^6 +\n |u|^2 + |u|^2|\na u|^2 \right)dx\ea\ee
where  \bnn\Phi\triangleq  \m|\nabla u|^2+(\m+\l)(\mbox{div}u)^2-2P\div u+(2\mu+\lambda)P^2
\enn
satisfies \be\la{p23} \Phi\ge \frac{\mu}{2} |\na u|^2 -C\int P^2\,dx.\ee

Consequently, Gronwall's inequality together with Lemma \ref{lem3.1} implies (\ref{3.4}). This finishes the proof of Lemma \ref{lem3.3}.

\end{proof}

Next lemma deals with $\na\dot{u}$.
\begin{lemma}\la{lem-2}
Under the condition \eqref{1.10-1-1}, it holds that for $0\leq
T<T^\ast$,
\be\la{lemm-1}
\int\r|\dot{u}|^2dx + \int_0^t\int|\na\dot{u}|^2dxdt\le C.
\ee
\end{lemma}
\begin{proof}
Make use the fact
\be
\dot{f}=f_t +\text{div}(f\otimes u)-f\text{div}u
\ee
 to obtain that
\be
\ba
&\rho\dot{u}_t + \rho u\cdot\na\dot{u} + \na P_t + \mbox{div}(\na P\otimes u)\\
&\quad = \mu[\triangle u_t + \mbox{div}(\triangle u\otimes u)] + (\lambda+\mu)[\nabla \mbox{div} u_t + \mbox{div}((\nabla \mbox{div} u)\otimes u)].
\ea
\ee
Multiplying the above equation by $\dot{u}$ and integrating over $R^3$ show that
\be\label{3.23}
\ba
\f12\f{d}{dt}\int\r|\dot{u}|^2dx =&\int P_t\rm{div}\dot{u} + (u\cdot\na\dot{u})\cdot\nabla Pdx +\m\int\dot{u}\cdot[\Delta{u}_t +\mbox{div}(\Delta{u}\otimes u)]dx \\
&+(\m+\l)\int\dot{u}\cdot[\na\mbox{div}u_t
+\mbox{div}((\na{\rm div}u)\otimes u)]dx= \sum_{i=1}^{3}N_i.
\ea
\ee
First, recalling that
\be
P_{t}+\mbox{div}(P u
)+P\mbox{div}u=2\mu|\mathfrak{D}(u)|^2+\l(\mbox{div}u)^2
\ee

One can get  after integration by parts and
using  the above equation that
\be\la{3.24-1} \ba
N_1 & =  \int   P_t\rm{div}\dot{u} + (u\cdot\na\dot{u})\cdot\nabla Pdx \\&=-\int div\dot{u}div(Pu) -\int Pdiv\dot{u}:div udx\\
& + \int\left(2\mu|\mathfrak{D}(u)|^2+\l(\mbox{div}u)^2\right)
div\dot{u}dx +\int(u\cdot\na\dot{u})\cdot\na Pdx\\
&\le \int P(u\cdot\na)div\dot{u}dx-\int P(u\cdot\na)div\dot{u}dx - \int P\na\dot{u}\na u^tdx\\
&+ C\left(\|\na u\|_{L^2}\|\na\dot{ u}\|_{L^2} +
\|\na u\|_{L^4}^2\|\na\dot{ u}\|_{L^2}\right)\\
&\le C\left(1+\|\na u\|_{L^4}^4\right) + \frac{\mu}{8}\|\na\dot{u}\|_{L^2}^2
\ea
\ee

Integration by parts leads to
\be\label{3.25} \ba N_2 & =  \mu\int
 \dot{u}_j[\pa_t\triangle u_j
 + \div (u\triangle u_j)]dx \\
& = - \mu\int  \left(\p_i\dot{u}_j(\p_iu_j)_t +
\triangle u_ju\cdot\nabla\dot{u}_j\right)dx \\
& = -  \mu\int \left(|\nabla\dot{u}|^2 -
\p_i\dot{u}_ju_k\p_k\p_iu_j - \p_i\dot{u}_j\p_iu_k\p_ku_j +
\triangle u_ju\cdot\nabla\dot{u}_j\right)dx \\
& = - \mu\int  \left(|\nabla\dot{u}|^2 + \p_i\dot{u}_j
\p_iu_j\div u - \p_i\dot{u}_j\p_iu_k\p_ku_j - \p_iu_j\p_iu_k\p_k\dot{u}_j
\right)dx \\
&\le -\frac{7\mu}{8} \int |\nabla\dot{u}|^2dx  + C \int
 |\nabla u|^4dx  . \ea \ee
Similarly,
\begin{eqnarray}\label{3.26}
N_3&\leq&
-\f78(\m+\l)\|\mbox{div}\dot{u}\|_{L^2}^2+C\int|\nabla{u}|^4dx.
\end{eqnarray}

Substituting \eqref{3.24-1}-\eqref{3.26} into \eqref{3.23}, we obtain after choosing $\ve$ suitably small that
\begin{equation}\label{3.22}\ba
& \f{d}{dt}\int\r|\dot{u}|^2dx+\m\|\nabla\dot{u}\|_{L^2}^2 \\&\leq C\int|\nabla{u}|^4dx+C.\ea
\end{equation}
Note that
\be
\ba
\|\na u\|_4^4 & \le \|\na u\|_2\|\na u\|_6^3\\
&\le C\|\na u\|_6^2(1+\|\na G\|_2+\|\na\omega\|_2)\\
&\le C\|\na u\|_6^2(1+\|\r^{\frac{1}{2}}\dot{u}\|_2)
\ea
\ee
Substituting this estimate into (\ref{3.22}) and once again recalling that
\be\la{u6}
\ba
\|\na u\|_6& \le\|G\|_6+\|\omega\|_6 + C\\
&\le C\|\rho\dot{u}\|_{L^2}+C\in L^2(0,T),
\ea
\ee
 we conclude by Gronwall's inequality that
\be
\int\r|\dot{u}|^2dx + \int_0^t\int|\na\dot{u}|^2dxdt\le C.
\ee

\end{proof}

Finally,  the following Lemma  \ref{lem3.7} will deal with the
higher order estimates of the solutions which are needed to
guarantee the extension of local strong solution to be a global one.
\begin{lemma}\label{lem3.7}
Under the condition \eqref{1.10-1-1},  it holds that  for $0\leq
T<T^*$,
\begin{equation}\label{3.54}
\sup_{0\leq t\leq T}(\|(\r,P)\|_{H^1\cap
W^{1, \ti q}}+\|\nabla{u}\|_{H^1}\leq C.
\end{equation}
\end{lemma}

\begin{proof}
 In view of (\ref{u6}) and (\ref{lemm-1}), one has
\be \la{c3.39} \|\na u\|_{L^6}\le C\|\n\dot u\|_{L^2} +C\le C.\ee

For $2\leq p\leq \tilde q<6,$  direct calculations show that
 \be \ba\label{a3.57}
\f{d}{dt}\|\nabla\r\|_{L^p} \leq& C(1+\|\nabla
u\|_{L^\infty})\|\nabla\r\|_{L^p}+C\|\nabla^2 u\|_{L^p}.
\ea\ee
Similarly,
\be\la{pr-1}
\f{d}{dt}\|\nabla P\|_{L^p} \leq C(1+\|\nabla
u\|_{L^\infty})(\|\nabla P\|_{L^p}+\|\na^2 u\|_{L^p})+C\|\nabla^2 u\|_{L^p}.
\ee

Applying the standard $L^p$-estimate to (\ref{mom}) gives
\bnn \|\na G\|_{L^6}+\|\omega\|_{L^6}\le C\|\n\dot u\|_{L^6}\le C+C\|\na\dot u\|_{L^2} ,\enn
which  shows \bnn \|G\|_{L^\infty}\le\|G\|_{L^2}^{\beta}\|\na G\|_{L^6}^{1-\beta}\le C+ C\|\na\dot u\|_{L^2}^{1-\beta}. \enn
for some $\beta\in (0,1)$.

Applying the standard $L^p$-estimate to \eqref{a1}$_2$ leads to \be\ba
 \la{3.59}\|\na^2 u\|_{L^p}  &\le   C\left(\|\n\dot u\|_{L^p}+
  \|\nabla P\|_{L^p}\right)\\ & \le C\left(1+\|\rho\dot{u}\|_2^\alpha\|\n\dot u\|_{L^6}^{1-\alpha}+\|\nabla P\|_{L^p}\right)
  \\ & \le C\left(1+\|\rho\dot{u}\|_2^\alpha\|\na\dot u\|_{L^2}^{1-\alpha}+\|\nabla P\|_{L^p}\right)\\
 & \le   C\left(1+\|\na\dot u\|_{L^2}^{1-\alpha}+\|\nabla P\|_{L^p}\right) \ea\ee
for some $\alpha\in (0,1)$.

This together with Lemma \ref{lem2.3} gives \be\la{3.43} \|\na u\|_{L^\infty}\le C(1+\|\na\dot{u}\|_{L^2}^{1-\beta}) \log \left(e+\|\na\dot{u}\|_{L^2} + \|\na P\|_{L^{\ti q}}  \right)+ C\|\na\dot u\|_{L^2}.\ee

 Substituting \eqref{3.59} and \eqref{3.43}  into \eqref{a3.57}-(\ref{pr-1}) yields that
\begin{equation}\la{a3.60}
f'(t)\leq  Cg(t)f(t)\ln{f(t)},
\end{equation} where
\bnn
 f(t)\triangleq e+\|\nabla\r\|_{L^{\ti q}}+\|\nabla P\|_{L^{\ti q}}, \,\, g(t)\triangleq
1+\|\nabla\dot{u}\|_{L^2}^2 .
\enn
It thus follows from  \eqref{a3.60}, \eqref{lemm-1},  and    Gronwall's inequality that
 \begin{equation}\label{3.64}
\sup_{0\leq t\leq T}\|(\nabla\r,\na P)\|_{L^{\ti q}}\leq C,
\end{equation}
which, combined with \eqref{3.43} and \eqref{lemm-1}
gives directly that
\begin{equation}\label{3.65}
\int_{0}^{T}\|\nabla u\|^2_{L^\infty}dt\leq C.
\end{equation}
Taking $p=2$ in \eqref{a3.57}, one can get by using \eqref{3.65}, \eqref{3.59} and Gronwall's inequality  that
\begin{equation}\label{3.66}
\sup_{0\leq t\leq T}\|(\nabla\r,\na P)\|_{L^2}\leq C,
\end{equation}
which together with  \eqref{3.59} yields that
\begin{eqnarray*}\label{3.67}
\sup_{0\leq t\leq T}\|\nabla^2u\|_{L^2}\leq C\sup_{0\leq t\leq
T}\left(\|\r\dot{u}\|_{L^2}+\|\nabla\r\|_{L^2}
+\|\nabla P\|_{L^2} \right) \leq C.
\end{eqnarray*}
This combined  with \eqref{3.64},  \eqref{3.66}, and
\eqref{3.4} finishes the proof of Lemma \ref{lem3.7}.

\end{proof}

  Now we are in a position to prove Theorem  \ref{thm1.2}.

{\it Proof  of Theorem \ref{thm1.2}.}
Suppose that \eqref{1.10-1-1} holds.
 Note that the generic constant  $C$ in Lemma
 \ref{lem3.7} remains uniformly bounded for all
$T<T^\ast$, so the functions
$(\r,u,P)(x,T^\ast)\triangleq\lim\limits_{t\rightarrow
T^\ast}(\r,u,P)(x,t)$ satisfy the conditions imposed on the initial
data  at the time $t=T^\ast$. Furthermore, standard
arguments yield that $\r\dot{u}\in C([0,T];L^2)$, which
implies \bnn \r\dot{u}(x,T^\ast)=\lim_{t\rightarrow
T^\ast}\r\dot{u}\in L^2. \enn Hence, \bnn
&&-\m\Delta{u}-(\m+\l)\nabla\mbox{div}u+\nabla P
|_{t=T^\ast}=\sqrt{\r}(x,T^\ast)g(x)
\enn with \bnn  g(x)\triangleq
\begin{cases}
\r^{-1/2}(x,T^\ast)(\r\dot{u})(x,T^\ast),&
\mbox{for}~~x\in\{x|\r(x,T^\ast)>0\},\\
0,&\mbox{for}~~x\in\{x|\r(x,T^\ast)=0\},
\end{cases}
\enn
 satisfying $g\in L^2$ due to \eqref{3.54}. Therefore, one can take $(\r,u,P)(x,T^\ast)$ as
the initial data and apply Lemma \ref{lem2.1} to extend the local
strong solution beyond $T^\ast$. This contradicts the assumption on
$T^{\ast}$. We thus finish the proof of Theorem \ref{thm1.2}.

\section{Outline of Theorem \ref{thm1.3}}
The main idea is quite analogous to section 3.
\begin{proof}

{\it Case I.} $\Omega=\mathbf{R}^3$ with non-vacuum far-field.

It's sufficient to note that
   \be\la{cc.2} (\n-\ti\n)_t+\div((\n-\ti\n)u)+ \ti \n\div u=0.\ee
   Multiplying \eqref{cc.2} by $\n-\ti\n$ and integrating the resulting equation over $\om,$ we obtain after using Lemma \ref{lem3.1} that \bnn (\|\n-\ti\n\|_{L^2}^2)'(t)\le C\|\n-\ti\n\|_{L^2}^2+C\|\na u\|_{L^2}^2,\enn
Therefore,
\be
\rho-\tilde\rho\in L^\infty L^2.
\ee
Similarly, one can show
\be
(\rho-\tilde\rho,P-\tilde P)\in L^\infty(0,t;L^1\cap L^\infty).
\ee
Then the remaining proof can be done step by step.\\

{\it Case II.} $\Omega=\mathbf{T}^3$.

We need only to redefine the effective flux $G$ as
\be
G=(2\mu+\lambda)\text{div}u-P + \bar P,
\ee
where
\be
\bar P=\frac{1}{|\mathbf{T}^3|}\int Pdx.
\ee
For any $q\in [1,\infty)$ and under the condition \eqref{1.10-1-1}, $G$ satisfies
\be
\|G\|_{L^q}\le C\|\na u\|_{L^q} + \|P\|_{L^q} + C
\ee
and
\be
\|\na G\|_{L^q} + \|\omega\|_{L^q}\le C\|\rho\dot{u}\|_{L^q}.
\ee
Then the higher regularity of the density, velocities and temperature can be obtained without difficulty.\\

{\it Case III.} $\Omega$ is a bounded domain.

Since there is no boundary condition for effective viscous flux. We will outline the proof of Lemma \ref{lem3.3}.

Motivated by \cite{wz1}, we decompose the velocity into two parts. It follows from Lemma \ref{lem2.3} that for any $t\in [0,T], $ there exists a unique  $v(t,\cdot) \in H^1_0\cap  W^{2,\ti q}$ satisfying
\be\label{lame}
 \mu\triangle v  + (\mu+\lambda)\nabla\mbox{div}v  = \nabla P  ,
\ee
which together with \eqref{b6}  yields that
\be \label{3.n3}\|\na v \|_{L^p}\le C\|P \|_{L^p}\le C, \mbox{ for } p\in [2,6],\,\, t\in [0,T],\ee and that
\be \label{3.n1} \ba
-\int P_t {\rm div} v dx
& = -  \int (\mu\na v_t\cdot\na v + (\mu+\lambda) \mbox{div}v_t\div v) dx \\
& =- \frac12 \frac{d}{dt} \int \left(\mu | \na  v |^2 +(\mu+\lambda)(\div v)^2\right)dx.
\ea \ee
Denoting by \be\la{dw} w\triangleq u-v,\ee   we have  $w \in H^1_0\cap W^{2,\ti q},$  for a.e. $t\in [0,T].$ Moreover, for a.e. $t\in [0,T],$ $w$  satisfies \be\la{wl1}
 \mu\triangle w + (\mu+\lambda)\nabla\mbox{div}w = \n \dot u,
\ee
which together with the standard $L^2$-estimate for elliptic system  gives
\be \label{3.n2}
\|\na w\|_{L^6}+\|\na^2 w\|_{L^2}\le C \|\n \dot u\|_{L^2}.
\ee

 It follows from \eqref{energy} and \eqref{en-1} that
\be \label{3.m1-1} \ba
   -\int P_t  {\rm div} w dx    & =- \int (\rho E )_t {\rm div} w dx+  \int (\rho |u|^2)_t {\rm div} w dx   \\
 & \le C   \int \left( \rho E  |u| +|u||\na u|\right) |\na^2w| dx\\
 &\quad   - \frac{1}{2}\int\left( {\rm div} (\rho u) |u|^2 {\rm div} w         - 2 \rho u\cdot u_t   {\rm div} w \right)dx   =\sum\limits_{i=1}^3J_i.
\ea \ee
In fact, \eqref{3.m1-1} has a same structure of \eqref{3.m1}. Here $\text{div}\omega$ plays a same role as $G$ in \eqref{3.m1}.

We then finish the proof of Theorem \ref{thm1.3} for bounded domain by adapting a same procedure as Theorem \ref{thm1.2} with the help of Lemma \ref{lem2.3}.

\end{proof}

{\bf Acknowledgment}. \  X.D. Huang is supported  by
President Fund of Academy of Mathematics Systems Science, CAS, No.2014-cjrwlzx-hxd and National Natural Science Foundation of China, No. 11471321 and 11371064. Zhouping Xin is partially supported by Zheng Ge Ru Foundation, Hong Kong RGC Earmarked Grants CUHK4041/11P and CUHK4048/13P, NSFC/RGC Joint Research Scheme Grant N-CUHK443/14, a grant from Croucher Foundation, and a Focus Area Grant at The Chinese University of Hong Kong.

\end{document}